\documentclass[a4paper,12pt]{amsart}

 \usepackage{amssymb,amsthm}
\usepackage{amscd} 
\usepackage{graphicx,color} 

\usepackage{url} 

\newtheorem{theorem}{Theorem}[section]

\newtheorem{lemma}[theorem]{Lemma}

\theoremstyle{definition}

\def\r{\mathbb R}

\def\n{\mathbb N}

\def\s{\mathbb S}

\begin{document}

\title[Stationary surfaces for   the moment of  inertia with constant Gauss curvature]{Stationary surfaces for   the moment of  inertia with constant Gauss curvature}
\author{Rafael L\'opez}
\address{ Departamento de Geometr\'{\i}a y Topolog\'{\i}a\\ Universidad de Granada. 18071 Granada, Spain}
\email{rcamino@ugr.es}

\begin{abstract}  Consider the energy $E_\alpha[\Sigma]=\int_\Sigma |p|^\alpha\, d\Sigma$, where $\Sigma$ is a surface in Euclidean space $\r^3$ and $\alpha\in\r$. We prove that planes and spheres are the only stationary surfaces for $E_\alpha$ with constant Gauss curvature. We also characterize these surfaces assuming that a principal curvature is constant or that the mean curvature is constant.\end{abstract}

\maketitle
\section{Introduction and statement of the results}


Euler investigated planar curves $\gamma$ of constant density joining two given points and having minimum moment of inertia with respect to the origin    \cite[p. 53]{eu}. In polar coordinates $r=r(\theta)$, these curves are minimum of the energy 
$$\int_\gamma r^\alpha\, ds=\int_{\theta_1}^{\theta_2} r^\alpha\sqrt{r'^2+r'^2}\ d\theta,$$
 where $\alpha$ is a parameter. In fact, Euler found the extremals of this energy for two explicit parameters. Indeed, if $\alpha=1/2$, the solution is an equilateral hyperbola. If $\alpha=2$, the solution is  $ar^3=\mbox{sec}(3\theta+b)$, $a,b\in\r$. In particular, it is necessary that the two points $\theta_1$ and $\theta_2$ subtend an angle of less than $\pi/3$ at the origin. For any $\alpha$, explicit parametrizations of the extremals can be found \cite{ca,to}.  The case $\alpha=-1$ was considered by Bonnet in \cite{bo}. For   $\alpha=2$, Mason proved that given two points of the extremal curve,   the extremal curve is an absolute minimum of the energy. In case that $|\theta_1-\theta_2|\geq \pi/3$, then the absolute minimum of the energy is the curve formed by the two lines between both points with the origin \cite{ma}. 

Recently, Dierkes and Huisken have extended this problem in arbitrary dimensions  \cite{dh}. In this paper we consider the case of    surfaces in Euclidean space $\r^3$. Let $\Sigma$ be a connected, oriented surface    and let  $\Phi\colon\Sigma\to\r^3$ be  an immersion of $\Sigma$ in    $\r^3$.  For $\alpha\in\r$, define the energy  
$$E_\alpha[\Sigma]=\int_\Sigma |\Phi|^\alpha\, d\Sigma,$$
where $d\Sigma$ denotes the area element induced on $\Sigma$. Critical points of $E_\alpha$ are called {\it $\alpha$-stationary surfaces} and we say  simply stationary surfaces if we do not want to emphasize the value of $\alpha$. Assuming that  $0\not\in\Phi(\Sigma)$, the Euler-Lagrange equation associated to $E_\alpha$ is  
\begin{equation}\label{eq1}
H=\alpha\frac{\langle N,\Phi\rangle}{|\Phi|^2},
\end{equation}
where $H$ and $N$ are the mean curvature and the unit normal vector of $\Sigma$, respectively. Notice that if $\alpha=0$ then  stationary surfaces are minimal surfaces. From now, we will discard the case $\alpha=0$. In \cite{dh}, the authors  investigated  stability of spheres and minimal cones as well as  minimizers of $E_\alpha$. See also \cite{cx,d4}. 

We find all $\alpha$-stationary surfaces in the family of isoparametric surfaces. Recall that an isoparametric surface is a surface with constant Gauss curvature and constant mean curvature. In $\r^3$, the isoparametric surfaces are   planes, spheres and circular cylinders. An easy computation shows the following 
\begin{enumerate}
\item Planes are stationary surfaces if and only if they contain the origin $0\in\r^3$. This holds for any $\alpha\in\r$.
\item The only stationary spheres  are the following:
\begin{enumerate}
\item    Spheres centered at $0$. Here  $\alpha=-2$.
\item    Spheres containing $0$. Here   $\alpha=-4$.  
\end{enumerate}
\item Circular cylinders are not stationary surfaces for all $\alpha\in\r$.
\end{enumerate}

In \cite{dh}, the authors focus on closed stationary surfaces, proving that if $\alpha>-2$, there are not closed stationary surfaces and if $\alpha<-2$, then the closure of any $\alpha$-stationary surface must contain $0$. Finally, if $\alpha=-2$, spheres are the only closed    stable surfaces: see \cite[Thm. 1.6]{dh}.   In fact, for $\alpha=-2$ the hypothesis on stability can be dropped because an argument of comparison with spheres centered at the origin together the maximum principle proves that   closed $-2$-stationary surfaces must be spheres. Definitively, spheres are the only closed $\alpha$-stationary surfaces and this only holds when $\alpha=-2$.  

In the class of stationary surfaces, we point out that the previous two types of examples have constant Gauss curvature $K$: planes with $K=0$ and spheres with  $K\not=0$.  The objective of this paper is to  prove the converse and  characterize these examples as the only  stationary surfaces with constant Gauss curvature.   

\begin{theorem} \label{t1} Planes and spheres are the only stationary surfaces   with   constant Gauss curvature.
\end{theorem}

 This theorem will be proved in Sects. \ref{s3} and \ref{s4} because we distinguish the case $K\not=0$ and $K=0$, respectively. 

Furthermore, two more results are proved, where we replace the hypothesis on  the Gauss curvature and  assuming that a principal curvature is constant or that the mean curvature is constant.

\begin{theorem} \label{t22} Planes and spheres are the only  stationary surfaces with a constant principal      curvature. 
\end{theorem}

\begin{theorem} \label{t3} Planes and spheres are the only   stationary surfaces with constant mean curvature.   
\end{theorem}

The proofs will be done in Sects. \ref{s5} and \ref{s6} respectively. The techniques use  the Gauss and Codazzi equations in local coordinates of a surface given by lines of curvature. This will be introduced in Sect. \ref{s2}. In some steps during  the proofs, it will be  required long and hard computations of derivatives as well as  manipulations with polynomials. In this paper we employ {\sc Mathematica}$^\copyright$ to compute these symbolic calculations \cite{wo}.

\section{Preliminaries}\label{s2}

In this section we show local computations of the curvature on a surface when the coordinates are given by lines of curvature in open sets   free of umbilic points. The next calculations and results hold for the proofs of all results of this paper.

Let $\Sigma$ be an orientable surface immersed in $\r^3$. Denote by  $\nabla$ and $\overline{\nabla}$ be the Levi-Civita connections of $\Sigma$ and
$\mathbb{R}^3$, respectively.  Denote by $N$ the Gauss map of $\Sigma$. Let $\mathfrak{X}(\Sigma)$ be the space of tangent vector fields of $\Sigma$. The  Gauss and Weingarten formulas are, respectively, 
\begin{equation}\label{gw}  
\begin{split}
    \overline{\nabla}_XY&=\nabla_XY+h(X,Y), \\
    \overline{\nabla}_XN&=-SX,
    \end{split}
\end{equation}
for all $X,Y\in\mathfrak{X}(\Sigma)$, where $h\colon T\Sigma\times T\Sigma\to T^\perp\Sigma$ and  $S\colon T\Sigma\to  T\Sigma$ are the second fundamental form and the shape
operator, respectively. The  Gauss
equation and Codazzi equations are 
\begin{equation*}
\begin{split}
    \langle R(X,Y)W,Z\rangle&=\langle h(Y,W),h(X,Z)\rangle-\langle
    h(X,W),h(Y,Z)\rangle\\
    (\nabla_XS)Y&=(\nabla_YS)X,
    \end{split}
\end{equation*}
where $R$ is the curvature tensor of $\nabla$. The   mean curvature  $H$ and the Gauss curvature $K$ are defined by
\[   H=\kappa_1+\kappa_2,\qquad K=\kappa_1\kappa_2,\]
where  $\kappa_1$ and $\kappa_2$ are the principal curvatures of $\Sigma$.

In an open set $\Omega\subset\Sigma$ of non-umbilic points, consider a local
orthonormal frame $\{e_1,e_2\}$ by line of curvature coordinates: see for example,  \cite[Ch. 3]{pt}. The expression of the shape operator $S$ is 
\begin{equation}\label{b0}
\left\{
\begin{split}
Se_1&=\kappa_1 e_1\\
Se_2&=\kappa_2 e_2.
\end{split}
\right.
    \end{equation}
Since $\{e_1,e_2\}$ is an orthonormal frame, if $X\in\mathfrak{X}(\Sigma)$ then
\begin{equation*}
\left\{
\begin{split}
\nabla_X e_1&=\omega(X)e_2\\
 \nabla_X e_2&=-\omega(X)e_1,
\end{split}
\right.
\end{equation*}
where $\omega(X)=\langle\nabla_X e_1,e_2\rangle=-\langle\nabla_X e_2,e_1\rangle$. Then we have 
\begin{equation}\label{eiej}
\left\{
\begin{split}
&\nabla_{e_1}e_1=\omega(e_1)e_2,\quad \nabla_{e_1}e_2=-\omega(e_1)e_1\\
&\nabla_{e_2}e_1=\omega(e_2)e_2,\quad \nabla_{e_2}e_2=-\omega(e_2)e_1.
\end{split}\right.
\end{equation}
Using the Codazzi equations, we have  
\begin{equation}\label{b1}
    \left\{
    \begin{array}{ll}
        e_{2}(\kappa_1)=(\kappa_1-\kappa_2)\omega(e_1)\\
        e_{1}(\kappa_2)=(\kappa_1-\kappa_2)\omega(e_2),
    \end{array}
    \right.
\end{equation}
where we stand for $e_i(f)$ the derivative of a smooth function $f\in C^\infty(\Omega)$ along the principal directions $e_i$, $i=1,2$.  In the set $\Omega$, we have $\kappa_1-\kappa_2\not=0$. The  Gauss equation  gives the following expression of the Gauss curvature  $K$,  
\begin{equation}\label{bb}
   K=\kappa_1\kappa_2=-e_1\left(\frac{e_1( \kappa_2 )}{\kappa_1-\kappa_2} \right)+e_2\left(\frac{e_2( \kappa_1 )}{\kappa_1 -\kappa_2}\right)
   -\frac{e_1( \kappa_2)^2+ e_2( \kappa_1 )^2}{(\kappa_1 -\kappa_2)^2}.
\end{equation}
 Denote by $\Phi\colon\Sigma\to\r^3$ the immersion of stationary surface. Let  decompose the position vector $\Phi$ in its tangent  $\Phi^\top$ and normal part $\Phi^\perp$ with respect to $\Sigma$. Let 
\begin{equation*}
\Phi^\top=\gamma e_1+\mu e_2,
\end{equation*}
for some smooth functions $\gamma$ and $\mu$. Notice that 
$$\gamma=\langle \Phi,e_1\rangle=\frac12e_1(|\Phi|^2),\quad \mu=\langle\Phi,e_2\rangle=\frac12e_2(|\Phi|^2).$$
 Analogously, the normal part  is  $\Phi^\perp=\langle N,\Phi\rangle N$, which by the stationary surface equation \eqref{eq1}, is
$$\Phi^{\perp}=\frac{H|\Phi|^2}{\alpha}N.$$
\begin{lemma}\label{l2}
Let $\Omega$ be a subset of non-umbilic points of an $\alpha$-stationary surface $\Sigma$. Then the functions   $\gamma$ and $\mu$ satisfy the following equations
    \begin{equation}\label{d3}
    \left\{
\begin{split}
            e_{1}(\gamma)&=\mu\omega(e_1)+\frac{H|\Phi|^2}{\alpha}\kappa_1+1\\
            e_{2}(\gamma)&=\mu\omega(e_2)\\
            e_{1}(\mu)&=-\gamma\omega(e_1)\\
            e_{2}(\mu)&=-\gamma\omega(e_2)+\frac{H|\Phi|^2}{\alpha}\kappa_2+1\\
            0&=e_{1}(\frac{H|\Phi|^2}{\alpha})+\gamma\kappa_1\\
            0&=e_{2}(\frac{H|\Phi|^2}{\alpha})+\mu\kappa_2.
        \end{split}\right.
    \end{equation}
\end{lemma}

\begin{proof} 
 If $X\in\mathfrak{X}(\Sigma)$, then   $\overline{\nabla}_X\Phi=X$. The Gauss and
Weingarten formulas \eqref{gw} imply 
\begin{equation*}
\begin{split}
    X&=\overline{\nabla}_{X}\Phi=\overline{\nabla}_{X}(\Phi^{\top}+\Phi^{\perp})=\overline{\nabla}_{X} \Phi^{\top}+\overline{\nabla}_X\Phi^{\perp}\\
    &={\nabla}_{X}\Phi^{\top}+h(X,\Phi^{\top})+X( H\frac{|\Phi|^2}{\alpha})N- H\frac{|\Phi|^2}{\alpha} SX.
\end{split}
\end{equation*}
Taking the tangent and the normal part of the above expression, we obtain, respectively,
\begin{equation}\label{nnn} 
    \left\{
    \begin{array}{ll}
        {\nabla}_{X}\Phi^{\top}-H\frac{|\Phi|^2}{\alpha}SX=X\\
        h(X, \Phi^{\top})+X(H\frac{|\Phi|^2}{\alpha})N=0.
    \end{array}
    \right.
\end{equation}
Equations \eqref{d3} are obtained by substituting in \eqref{nnn} $X$ by $e_1$ and   $e_2$ in the previous two identities and, next, writing the  coordinates with respect to $\{e_1,e_2,N\}$. In the case $X=e_1$, and using \eqref{eiej}, we have
\begin{equation*} 
    \begin{array}{ll}
       e_1&=  {\nabla}_{e_1}(\gamma e_1+\mu e_2)-H\frac{|\Phi|^2}{\alpha}\kappa_1 e_1\\
     & =e_1(\gamma)e_1+e_1(\mu)e_2+\gamma \omega(e_1)e_2-\mu\omega(e_1)e_1-H\frac{|\Phi|^2}{\alpha}\kappa_1 e_1.
     \end{array}
\end{equation*}
\begin{equation*} 
    \begin{array}{ll}
       0&= h(e_1,\gamma e_1+\mu e_2)+e_1(H\frac{|\Phi|^2}{\alpha})N\\
       &=\left(\gamma\kappa_1+e_1\left(\frac{ H|\Phi|^2}{\alpha}\right)\right)N.
    \end{array}
\end{equation*}
This gives the first, third and fifth equations in \eqref{d3}. The other three equations of \eqref{d3} are obtained by putting $X=e_2$ into \eqref{nnn}.
\end{proof}

\section{Proof of Theorem \ref{t1}: case $K\not=0$}\label{s3}

Let $\Sigma$ be an $\alpha$-stationary surface and suppose that $K$ is constant with $K\not=0$. 
If a principal curvature is constant, then the other one is also constant because $K$ is constant and $K\not=0$. Then $\Sigma$ is an isoparametric surface, hence  $\Sigma$ is a   sphere.  This proves the theorem in this situation. 

From now we suppose that both principal curvatures are not constant and we will arrive to a contradiction. We express $\kappa_2$, $H$ and $\kappa_1-\kappa_2$ in terms of $K$ and $\kappa_1$. We have 
\begin{equation}\label{gm2}
\begin{split}
\kappa_2&=\frac{K}{\kappa_1},\\
H&=\frac{\kappa_1^2+K}{\kappa_1},\\
\kappa_1-\kappa_2&=\frac{\kappa_1^2-K}{\kappa_1}.
\end{split}
\end{equation}
Using these identities and the fact that $K$ is constant, the expression of $K$ in \eqref{bb} can be written as 
 \begin{equation}\label{bb2}
  \begin{split}
   K\frac{(\kappa_1^2-K)^2}{\kappa_1^2}&=\frac{K(\kappa_1^2-K)}{\kappa_1^3}e_{11}(\kappa_1) +\frac{\kappa_1^2-K}{\kappa_1}e_{22}(\kappa_1)\\
   &-\frac{3K}{\kappa_1^2}e_1(\kappa_1)^2-\frac{2\kappa_1^2+K}{\kappa_1^2}e_2(\kappa_1)^2.
   \end{split}
\end{equation}

By using the expressions of $H$ and $\kappa_2$ in \eqref{gm2},  we have
\begin{equation}\label{gm22}
\begin{split}
e_1(\frac{H|\Phi|^2}{\alpha})&= \frac{\kappa_1^2-K}{\alpha\kappa_1^2}|\Phi|^2e_1(\kappa_1)+2\gamma\frac{\kappa_1^2+K}{\alpha \kappa_1},\\
e_2(\frac{H|\Phi|^2}{\alpha})&= \frac{\kappa_1^2-K}{\alpha\kappa_1^2}|\Phi|^2e_2(\kappa_1)+2\mu\frac{\kappa_1^2+K}{\alpha \kappa_1}.\\
\end{split}
\end{equation}
Substituting into  the last two equations of \eqref{d3}, and  using \eqref{gm2} and \eqref{gm22}, we have
\begin{equation*}
\begin{split}
 0&=\frac{\kappa_1^2-K}{\alpha\kappa_1^2}|\Phi|^2e_1(\kappa_1)+\left(2\frac{\kappa_1^2+K}{\alpha \kappa_1} + \kappa_1\right)\gamma,\\
0 &= \frac{\kappa_1^2-K}{\alpha\kappa_1^2}|\Phi|^2e_2(\kappa_1)+\left(2\frac{\kappa_1^2+K}{\alpha \kappa_1}+ \frac{K}{\kappa_1}\right)\mu.\\
\end{split}
\end{equation*}
From both identities, we get $\gamma$ and $\mu$, 
\begin{equation}\label{gm}
\begin{split}
 \gamma&=\frac{|\Phi|^2(K-\kappa_1^2)}{2\kappa_1 K+(\alpha +2)\kappa_1^3}e_1(\kappa_1) ,\\
\mu&=\frac{|\Phi|^2(K-\kappa_1^2)}{2\kappa_1^3+(\alpha +2)\kappa_1 K}e_2(\kappa_1).\\
\end{split}
\end{equation}
We point out that the denominators of \eqref{gm} cannot vanish identically because in such a case, $\kappa_1$ would be a constant function, which was initially discarded. This argument also holds in the reasonings in all the denominators of the next expressions.

 Using these above expressions of $\gamma$ and $\mu$, we differentiate with respect to $e_1$ and $e_2$. Later, we will compare with the first equations in \eqref{d3} and, hence, we will get the expressions of $e_{ij}(\kappa_1)$ in terms of $e_1(\kappa_1)$ and $e_2(\kappa_1)$.  We differentiate $\gamma$ and $\mu$ given in \eqref{gm} with respect to $e_1$ and $e_2$. Notice that  $e_1(|\Phi|^2)= 2\gamma$ and  $e_2(|\Phi|^2)= 2\mu$. Then we have
\begin{equation}\label{d4}
\begin{split}
 e_1(\gamma)&=\frac{|\Phi|^2(K-\kappa_1^2)}{2\kappa_1 K+(\alpha +2)\kappa_1^3}e_{11}(\kappa_1)+\frac{(\alpha +4)   |\Phi|^2 \left(\kappa_1^2-3 K\right)}{\left((\alpha +2) \kappa_1^2+2 K\right)^2}e_1(\kappa_1)^2\\ 
 e_1(\mu)&= \frac{|\Phi|^2(K-\kappa_1^2)}{2\kappa_1^3+(\alpha +2)\kappa_1 K}e_{12}(\kappa_1)+\frac{(\alpha +4) |\Phi|^2 \left(\kappa_1^2+K\right) \left(2 k^2-(\alpha +6) K\right)}{\left((\alpha +2)\kappa_1^2+2 K\right) \left(2 \kappa_1^2+(\alpha +2) K\right)^2}e_1(\kappa_1)e_2(\kappa_1)\\
 e_2(\mu)&=\frac{|\Phi|^2(K-\kappa_1^2)}{2\kappa_1^3+(\alpha +2)\kappa_1 K}e_{22}(\kappa_1)+\frac{  |\Phi|^2 \left(4 \kappa_1^4-(\alpha +12) \kappa_1^2 K-\alpha  K^2\right)}{\left(2 \kappa_1^3+(\alpha +2) \kappa_1 K\right)^2}e_2(\kappa_1)^2.
 \end{split}
 \end{equation}
We work on \eqref{d3}. Now \eqref{b1} gives, in combination with \eqref{gm2}, 
\begin{equation*}
\begin{split}
\omega(e_1)&=\frac{\kappa_1}{\kappa_1^2-K}e_2(\kappa_1),\\
\omega(e_2)&=-\frac{K}{\kappa_1(\kappa_1^2-K)}e_1(\kappa_1).\\
\end{split}
\end{equation*}
Then the first equations of \eqref{d3} become
    \begin{equation}\label{d5}
    \left\{
\begin{split}
e_{1}(\gamma)&=\mu\frac{\kappa_1}{\kappa_1^2-K}e_2(\kappa_1)+\frac{\kappa_1^2+K}{\alpha }|\Phi|^2  +1\\
 e_{1}(\mu)&=-\gamma\frac{\kappa_1}{\kappa_1^2-K}e_2(\kappa_1)\\
 e_{2}(\mu)&=\gamma \frac{K}{\kappa_1(\kappa_1^2-K)}e_1(\kappa_1)+\frac{\kappa_1^2+K}{\alpha\kappa_1^2}|\Phi|^2 K+1,
      \end{split}\right.
   \end{equation}
where we have used \eqref{gm2} again.  From \eqref{d4} and \eqref{d5}, we obtain the expressions of $e_{11}(\kappa_1)$, $e_{12}(\kappa_1)$ and $e_{22}(\kappa_1)$, namely,

\begin{equation}\label{d7}
\begin{split}
e_{11}(\kappa_1)&=\frac{(\alpha +4) \kappa_1 \left(\kappa_1^2-3 K\right)}{\left(\kappa_1^2-K\right) \left((\alpha +2) \kappa_1^2+2 K\right)}e_1(\kappa_1)^2+\frac{(\alpha +2) \kappa_1^3+2 \kappa_1 K}{\left(\kappa_1^2-K\right) \left(2 \kappa_1^2+(\alpha +2) K\right)}e_2(\kappa_1)^2\\
&-\frac{\kappa_1 \left((\alpha +2) \kappa_1^2+2 K\right) \left(\alpha +|\Phi|^2 \left(\kappa_1^2+K\right)\right)}{\alpha  |\Phi|^2 \left(\kappa_1^2-K\right)},\\
e_{12}(\kappa_1)&=\left( \kappa_1  \left(\frac{4}{2 \kappa_1^2+(\alpha +2) K}+\frac{3}{K-\kappa_1^2}\right)-\frac{4 K}{\kappa_1((\alpha +2) \kappa_1^2+2 K)}+\frac{2}{\kappa_1}\right)e_1(\kappa_1)e_2(\kappa_1),\\
e_{22}(\kappa_1)&=\frac{K \left(2 \kappa_1^3+(\alpha +2) \kappa_1 K\right)}{\left(\kappa_1^2-K\right) \left((\alpha +2) \kappa_1^4+2 \kappa_1^2 K\right)}e_1(\kappa_1)^2-\frac{-4 \kappa_1^4+(\alpha +12) \kappa_1^2 K+\alpha  K^2}{\left(\kappa_1^2-K\right) \left(2 \kappa_1^3+(\alpha +2) \kappa_1 K\right)}e_2(\kappa_1)^2\\
&-\frac{\left(2 \kappa_1^2+(\alpha +2) K\right) \left(|\Phi|^2 K \left(\kappa_1^2+K\right)+\alpha  \kappa_1^2\right)}{\alpha  |\Phi|^2 \kappa_1 \left(\kappa_1^2-K\right)}.
\end{split}
\end{equation}
We distinguish the   that $e_1(\kappa_1)=0$ identically and $e_1(\kappa_1)\not=0$.
\begin{enumerate}
\item Case  $e_1(\kappa_1)=0$ identically. Then $e_{11}(\kappa_1)=0$ identically and the first equation in \eqref{d7} gives 
\begin{equation}\label{E1}
e_2(\kappa_1)^2=\frac{\left(2 \kappa_1^2+(\alpha +2) K\right) \left(\alpha +|\Phi|^2 \left(\kappa_1^2+K\right)\right)}{\alpha  |\Phi|^2}.
\end{equation}
Notice that $e_2(\kappa_1)\not=0$ because $\kappa_1$ is not a constant function. 
Differentiating \eqref{E1} with respect to $e_2$, and simplifying by $e_2(\kappa_1)$, we have
$$e_{22}(\kappa_1)=\left(\frac{3 \kappa_1^2-K}{\kappa_1|\Phi|^2}+\frac{4 \kappa_1^3+  (\alpha +4) \kappa_1 K}{\alpha }\right).$$
On the other hand, the last equation in \eqref{d7} is now 
\begin{equation*}
\begin{split}
e_{22}(\kappa_1)&=\frac{4 \kappa_1^4-(\alpha +12) \kappa_1^2 K-\alpha  K^2}{\left(\kappa_1^2-K\right) \left(2 \kappa_1^3+(\alpha +2) \kappa_1 K\right)}e_2(\kappa_1)^2 \\
 &-\frac{\left(2 \kappa_1^2+(\alpha +2) K\right) \left(|\Phi|^2 K \left(\kappa_1^2+K\right)+\alpha  \kappa_1^2\right)}{\alpha  |\Phi|^2 \kappa_1 \left(\kappa_1^2-K\right)}
 \end{split}
 \end{equation*}
 because $e_1(\kappa_1)=0$ identically. Equating the two above expressions of $e_{22}(\kappa_1)$, and using the value of $e_{2}(\kappa_1)^2$ in \eqref{E1}, we get 
\begin{equation*}
\begin{split}
0&=2 |\Phi|^2 K \left((\alpha +5) \kappa_1^4+(\alpha +6) \kappa_1^2 K+(\alpha +1) K^2\right)\\
&+\alpha  \left(\kappa_1^4+2 (\alpha +5) \kappa_1^2 K+(\alpha +1) K^2\right). 
\end{split}
 \end{equation*}
 If the coefficient of $|\Phi|^2$ in this equation is $0$, then $(\alpha +5) \kappa_1^4+(\alpha +6) \kappa_1^2 K+(\alpha +1) K^2=0$. Since the coefficients are constant and they do not vanish identically, then $\kappa_1$ would be a constant function, which it is not possible. Therefore, we get  an expression for $|\Phi|^2$, namely, 
\begin{equation}\label{d10}
|\Phi|^2=-\frac{\alpha  \left(\kappa_1^4+2 (\alpha +5) K\kappa_1^2 +(\alpha +1) K^2\right)}{2 K \left((\alpha +5) \kappa_1^4+(\alpha +6) \kappa_1^2 K+(\alpha +1) K^2\right)}.
\end{equation}
Notice that we are using $K\not=0$. Differentiating with respect to $e_2$ and because $e_2(|\Phi|^2)=2\mu$, we obtain after some manipulations together the expression of $e_2(\kappa_1)^2$ given in \eqref{E1}, a polynomial on $\kappa_1$ of degree $10$ with constant coefficients, namely,  
\begin{equation*}
\begin{split}
0&=- (\alpha +5) \kappa_1^{10} +(2 \alpha  (\alpha +9)+37) K\kappa_1^8 +2 (\alpha  (\alpha  (\alpha +13)+47)+47) K^2\kappa_1^6 \\
&+2 (\alpha  (\alpha  (\alpha +6)+14)+29) K^3\kappa_1^4 -(\alpha +1) (\alpha  (\alpha +4)-7) K^4\kappa_1^2 \\
&+(\alpha +1)^2 K^5.
\end{split}
\end{equation*}
Since the coefficients of this polynomial cannot vanish mutually, this implies that $\kappa_1$ is a constant, obtaining a contradiction.
 \item Case $e_1(\kappa_1)\not=0$.
 We substitute \eqref{d7} in \eqref{bb2}, obtaining
\begin{equation}\label{pe1}
P_1e_1(\kappa_1)^2+Q_1e_2(\kappa_1)^2+R_1=0,
\end{equation}
where
\begin{equation*}
\begin{split}
P_1&=\frac{K \left(\alpha  \kappa_1^2+(\alpha +8) K\right)}{(\alpha +2) \kappa_1^4+2 k^2 K},\\
Q_1&=\frac{K \left((\alpha +8) \kappa_1^2+\alpha  K\right)}{2 \kappa_1^4+(\alpha +2) \kappa_1^2 K},\\
R_1&=  \frac{\alpha  \kappa_1^2 K+\left(\kappa_1^2+K\right)^2}{\kappa_1^2|\Phi|^2}+\frac{K}{\kappa_1^2} \left(\kappa_1^4+\frac{2 \left(\kappa_1^2+K\right)^2}{\alpha }+K^2\right) .
\end{split}
\end{equation*}
We differentiate \eqref{pe1} with respect to $e_1$, 
$$e_1(P_1)e_1(\kappa_1)^2+e_1(Q_1)e_2(\kappa_1)^2+e_1(R_1)+2P_1e_1(\kappa_1)e_{11}(\kappa_1)+2Q_1 e_2(\kappa_1)e_{12}(\kappa_1)=0.$$
Simplifying by $e_1(\kappa_1)$ because $e_1(\kappa_1)\not=0$, we obtain an equation of type
\begin{equation}\label{pe2}
P_2e_1(\kappa_1)^2+Q_2e_2(\kappa_1)^2+R_2=0,
\end{equation}
where
\begin{equation*}
\begin{split}
P_2&=2K(\kappa_1^6-3\kappa_1^4 K-4K^2\kappa_1^2+2K^3 ),\\
Q_2&= 2K^2(6\kappa_1^4+4K  \kappa_1^2+K^2),\\
R_2&=-\frac{\kappa_1^2 \left(\kappa_1^8 +K^4+2\kappa_1^2 K^3(4-3K|\Phi|^2)-2\kappa_1^4K^2(1+K|\Phi|^2)\right)}{|\Phi|^2}.
\end{split}
\end{equation*}
We solve Eqs. \eqref{pe1} and \eqref{pe2} for $e_1(\kappa_1)^2$ and $e_2(\kappa_1)^2$. First we need to discuss the case $P_1Q_2-P_2 Q_1=0$. We have  
$$P_1Q_2-P_2 Q_1=\frac{4 \alpha  \left(16-\alpha ^2\right) K^2 \left(\kappa_1^2-K\right)^4}{\left((\alpha +2) \kappa_1^2+2 K\right)^2 \left(2 \kappa_1^2+(\alpha +2) K\right)^2}.$$
Since $\kappa_1$ is not constant and $K\not=0$, the discussion is   $\alpha^2-16=0$ and $\alpha^2-16\not=0$.
\begin{enumerate}
\item Case $\alpha^2=16$. Suppose $\alpha=-4$ and the argument for $\alpha=4$ is similar. Combining Eqs. \eqref{pe1} and \eqref{pe2} by means of $P_2\eqref{pe1}-P_1 \eqref{pe2}=0$, we obtain 
$$2 \kappa_1^2 K \left(2 |\Phi|^2 K-3\right)-3 \kappa_1^4+K^2=0.$$
This gives 
$$|\Phi|^2=\frac{1}{4} \left(\frac{3 \kappa_1^2}{K^2}-\frac{1}{\kappa_1^2}+\frac{6}{K}\right).$$
Differentiating with respect to $e_1$, and using $e_1(|\Phi|^2)=2\gamma$ together the value of $\gamma$ in \eqref{gm}, we obtain 
$\kappa^2-K=0$, which implies that   $\kappa_1$ is constant. This is  a contradiction.
\item Case $\alpha^2 \not=16$. From Eqs. \eqref{pe1} and \eqref{pe2}, we obtain $e_1(\kappa_1)^2$ and $e_2(\kappa_2)^2$. The value of $e_1(\kappa_1)^2$ is 
\begin{equation}\label{E2}
e_1(\kappa_1)^2:=M_1=\frac{2K+(\alpha +2)\kappa_1^2}{2\alpha^2(16-\alpha^2)|\Phi|^2\kappa_1^2(\kappa_1^2-K)^3}M_2,
\end{equation}
where
\begin{equation*}
\begin{split}
M_2&=2 |\Phi|^2 \kappa_1^2 K \Big(\left(\alpha ^3+2 \alpha ^2-16 \alpha -32\right) \kappa_1^8-2 \left(2 \alpha ^3+11 \alpha ^2+32 \alpha +48\right) \kappa_1^6 K\\
&+\left(7 \alpha ^3+26 \alpha ^2+16 \alpha -96\right) \kappa_1^4 K^2+2 \left(\alpha ^4+10 \alpha ^3+39 \alpha ^2+32 \alpha -16\right) \kappa_1^2 K^3\\
&+\alpha ^2 \left(\alpha ^2+8 \alpha +12\right) K^4\Big)\\
&+\alpha  \Big(-6 (3 \alpha +8) \kappa_1^{10}+\left(3 \alpha ^3-\alpha ^2-110 \alpha -256\right) \kappa_1^8 K\\
&+\left(\alpha ^4+\alpha ^3-54 \alpha ^2-236 \alpha -352\right) \kappa_1^6 K^2-\left(3 \alpha ^3+32 \alpha ^2+148 \alpha +128\right) \kappa_1^4 K^3\\
&-\left(\alpha ^3+10 \alpha ^2+2 \alpha -16\right) \kappa_1^2 K^4+\alpha  (\alpha +2) K^5\Big).
\end{split}
\end{equation*}
We omit the expression of $e_2(\kappa_1)^2$. We   explain the argument that allows to get a contradiction.  In the expression \eqref{E2} of $e_1(\kappa_1)^2$, we differentiate with respect to $e_1$, obtaining 
\begin{equation}\label{efinal}
2e_1(\kappa_1)e_{11}(\kappa_1)=\frac{d}{\kappa_1}(M_1)e_1(\kappa_1).
\end{equation}
In this equation, we simplify by $e_1(\kappa_1)$. We employ the expressions of $e_1(\kappa_1)^2$ and $e_2(\kappa_1)^2$ in       \eqref{efinal} obtaining finally  
\begin{equation*}
\begin{split}
0&=2 \kappa_1^4 K^2 \left(2 (\alpha +3) (\alpha +4) |\Phi|^2 K-(\alpha +1) (\alpha +7)\right)\\
&+(\alpha +1) (\alpha +3) \kappa_1^8-2 (\alpha +1) (\alpha +2) \kappa_1^6 K+2 (\alpha -2) \kappa_1^2 K^3+3 K^4.
\end{split}
\end{equation*}
As in the case $\alpha^2=16$, we now get $|\Phi|^2$ from this identity. The coefficient of $|\Phi|^2$ is $4 (\alpha +3) (\alpha +4) \kappa_1^4 K^3$. If $\alpha=-3$ (notice that $\alpha\not=-4$), then we have 
$$4 \kappa_1^6-16 \kappa_1^4 K+10 \kappa_1^2 K^2-3 K^3=0.$$
 This proves that $\kappa_1$ is constant and this is a contradiction. If $\alpha\not=-3$, then
$$|\Phi|^2:=M_3=\frac{M_4}{4 (\alpha +3) (\alpha +4) \kappa_1^4 K^3},$$
where
\begin{equation*}
\begin{split}
M_4&=- \left(\alpha ^2+4 \alpha +3\right) \kappa_1^8-2 \left(\alpha ^2+3 \alpha +2\right) \kappa_1^6 K\\
&-2 \left(\alpha ^2+8 \alpha +7\right) \kappa_1^4 K^2+2 (\alpha -2) \kappa_1^2 K^3+3 K^4.
\end{split}
\end{equation*}
Again, we differentiate with respect to $e_1$ and we use $e_1(|\Phi|^2)=2\gamma$, obtaining 
$$2\gamma=\frac{d}{\kappa_1}(M_3) e_1(\kappa_1).$$
This gives a polynomial of degree $8$, 
\begin{equation*}
\begin{split}
0&=(\alpha +1) (\alpha +3) (2 \alpha +5) \kappa_1^8+4 (\alpha +1) (\alpha +3)K \kappa_1^6 -6 (\alpha +1)K^2 \kappa_1^4 \\
&+12 (\alpha +1) K^3\kappa_1^2 +15 K^4.
\end{split}
\end{equation*}
Since the coefficients are constant and not mutually zero, we deduce that $\kappa_1$ is a constant function, obtaining a contradiction. This completes the proof of Thm. \ref{t1}.
\end{enumerate}
 
 \end{enumerate}

\section{Proof of Theorem \ref{t1}: case $K=0$}\label{s4}

Let $\Sigma$ be an $\alpha$-stationary surface with $K=0$.   Without loss of generality, assume   $\kappa_2=0$. If $\kappa_1$ is constant,  then the surface is isoparametric, hence $\Sigma$ is a plane because $K=0$. This proves the result in this situation.
 
From now on, we suppose that $\kappa_1$ is not constant and we will arrive to a contradiction. Now the expression of $\mu$ is \begin{equation}\label{gm4}
\mu=-\frac{|\Phi|^2}{2\kappa_1}e_2(\kappa_1).
\end{equation}
  From \eqref{b1} and   \eqref{d3}, we obtain $\omega(e_2)=0$  and $e_2(\mu)=1$.   On the other hand, from \eqref{gm4}, we have 
\begin{equation*}
e_2(\mu)=-\frac{|\Phi|^2}{2\kappa_1}e_{22}(\kappa_1)+\frac{|\Phi|^2}{\kappa_1^2}e_2(\kappa_1)^2.
  \end{equation*}
  Since $e_2(\mu)=1$, then
  $$e_{22}(\kappa_1)=\frac{2}{\kappa_1} e_2(\kappa_1)^2-\frac{2 \kappa_1}{|\Phi|^2}.$$
 Equation \eqref{bb} becomes 
  $$e_{22}(\kappa_1)=\frac{2}{\kappa_1}e_2(\kappa_1)^2.$$
Combination the last two expressions of $e_{22}(\kappa_1)$, we obtain $\kappa_1=0$, which it is a contradiction.

\section{Proof of Theorem \ref{t22}}\label{s5}
 
 Let $\Sigma$ be an $\alpha$-stationary surface with a constant principal curvature.    Without loss of generality, we assume   $\kappa_2=c$. If $c=0$, then Thm. \ref{t1} proves that $\Sigma$ is a plane. 
 
 Suppose now $c\not=0$. If the other principal curvature $\kappa_1$ is constant,  then the surface is isoparametric, hence $\Sigma$ is a plane or a sphere. This proves the result in this situation. 
 
 We now assume that $\kappa_1$ is not constant and we will obtain a contradiction. The analogue identities \eqref{gm2} are  
 \begin{equation*}
 \begin{split}
 K&=\frac{c}{\kappa_1}\\
 H&=\kappa_1+c\\
 \kappa_1-\kappa_2&=\kappa_1-c.
 \end{split}
 \end{equation*}
 On the other hand, the expression of $K$ in \eqref{bb} becomes
  \begin{equation}\label{f1}
 e_{22}(\kappa_1)=\frac{2}{\kappa_1-c}e_2(\kappa_1)^2+\kappa_1c(\kappa_1-c).
 \end{equation}
The last equation of \eqref{d3} yields
 \begin{equation}\label{f3}
 \mu=-\frac{|\Phi|^2}{2\kappa_1+(\alpha +2)c}e_2(\kappa_1).
 \end{equation}
 Differentiating this identity with respect to $e_2$, we have
\begin{equation*}
e_2(\mu)=-\frac{|\Phi|^2}{2\kappa_1+(\alpha +2)c}e_{22}(\kappa_1)+\frac{4|\Phi|^2}{(2\kappa_1+c(\alpha +2))^2}e_2(\kappa_1)^2.
\end{equation*}
The fourth equation of \eqref{d3} is
$$e_2(\mu)=1+\frac{(\kappa_1+c)|\Phi|^2}{\alpha}.$$
From the last two expressions of $e_2(\mu)$, we get $e_{22}(\kappa_1)$, namely,  
\begin{equation}\label{f2}
e_{22}(\kappa_1)=(2 \kappa_1+(\alpha +2) c)\left(\frac{4 e_2(\kappa_1)^2}{(2 \kappa_1+(\alpha +2) c)^2} - \frac{\kappa_1+c}{\alpha }-\frac{1}{|\Phi|^2 }\right).
\end{equation}
  Equating $e_{22}(\kappa_1)$ from \eqref{f1} and \eqref{f2}, we get   
\begin{equation}\label{E3}
\begin{split}
&(2 \alpha c (\alpha +4) |\Phi|^2)e_2(\kappa_1)^2\\
&= (c-\kappa_1) (2 \kappa_1+(\alpha +2) c) \Big(|\Phi|^2 \left(\alpha  c \left(\kappa_1^2-\kappa_1 c+\kappa_1+c\right)+2 (\kappa_1+c)^2\right)\\
&+\alpha  (2 \kappa_1+(\alpha +2) c)\Big).
\end{split}
\end{equation}
\begin{enumerate}
\item Case $\alpha+4=0$. Then  \eqref{E3} is simplified  into
 $$4+|\Phi|^2 ( (2c-1)\kappa_1-c)=0,$$
 hence 
 \begin{equation}\label{f4}
 |\Phi|^2=\frac{4}{(1-2c)\kappa_1+c}.
 \end{equation}
 Differentiating with respect to $e_2$, we have 
 $$2\mu=\frac{8c-4}{((1-2c)\kappa_1+c)^2}e_2(\kappa_1).$$
 Simplifying again by $e_2(\kappa_1)\not=0$, using \eqref{f3} and the value of $|\Phi|^2$ in \eqref{f4}, we have $\kappa_1-c+1=0$. This gives that $\kappa_1 $ is a constant, which it is a contradiction. 
 \item Case $\alpha\not=-4$. The argument   is  similar. From \eqref{E3}, the expression of $e_2(\kappa_1)^2$ is  
\begin{equation}\label{E4}
e_2(\kappa_1)^2:=M_2=\frac{(\kappa_1-c)((\alpha +2) c+2 \kappa_1)}{2 \alpha c  (\alpha +4)   |\Phi|^2}M_1,
\end{equation}
 where
\begin{equation*}
\begin{split}
M_1&=   c^2 |\Phi|^2 (\alpha  (\kappa_1-1)-2)-2 \kappa_1 \left(\alpha +|\Phi|^2 k\right) \\
&-c \left(\alpha  (\alpha +2)+|\Phi|^2\kappa_1 (\alpha +\alpha  \kappa_1+4)\right).
\end{split}
\end{equation*}
We now differentiate with respect to $e_2$, obtaining
$$2e_2(\kappa_1)e_{22}(\kappa_1)=e_2(\kappa_1)\frac{d}{d\kappa_1}(M_2).$$
Simplifying by $e_2(\kappa_1)$ and using the expressions of $e_{22}(\kappa_1)$ in \eqref{f2} and of $e_2(\kappa_1)^2$ in \eqref{E4},   we get 
$$|\Phi|^2M_3= 3 \alpha  (\alpha +4) ((\alpha +2) c+2 \kappa_1),$$
where 
\begin{equation*}
\begin{split}
M_3&=\alpha  (\alpha +2) c^3-4 (\alpha +3) c^2 (\alpha +\alpha  \kappa_1+2)\\
&+c \kappa_1 \left(\alpha ^2 (3 \kappa_1-2)+2 \alpha  (5 \kappa_1-12)-48\right)-4 (\alpha +6) \kappa_1^2.
\end{split}
\end{equation*}
Notice that $M_3$ cannot vanish identically because this would yield that $\kappa_1$ is a constant function. Thus we get the value of $|\Phi|^2$, namely,
$$|\Phi|^2=\frac{3 \alpha  (\alpha +4) ((\alpha +2) c+2 \kappa_1)}{M_3}.$$
Differentiating with respect to $e_2$, and because $e_2(|\Phi|^2)=2\mu$, we obtain 
\begin{equation*}
\begin{split}
0&=(\alpha +2) c^2 (\alpha  (c-4)-12)+\left(2 c (-\alpha  (\alpha +12)-2 \alpha  (\alpha +3) c-24)\right)\kappa_1\\
&+\left(\alpha  (3 \alpha +10) c-4 (\alpha +6)\right)\kappa_1^2.
\end{split}
\end{equation*}
This polynomial of degree $2$ on $\kappa_1$ has constant coefficients and they are not mutually identically $0$ because $\alpha\not=-4$. This proves that $\kappa_1$ is a constant function, obtaining a contradiction. This completes the proof of Thm. \ref{t22}.
 \end{enumerate}
\section{Proof of Theorem \ref{t3}}\label{s6}

  Let $\Sigma$  be an $\alpha$-stationary surface with constant mean curvature.  If a principal curvature is constant, then the other one is also constant because $H$ is constant. In such a case, $\Sigma$ is an isoparemetric surface. This implies that $\Sigma$ is   a plane or a sphere, proving the result in this situation. 
  
  Suppose now that $\kappa_1$ and $\kappa_2$ are not constant and we will arrive to a contradiction. In order to simplify the notation, we identify $\Phi(p)$ with $p$, $p\in\Sigma$.  Let  $\Omega\subset\Sigma$ be an open set  formed by points $p$ where $ N(p)\not=\pm \frac{p}{|p|}$ and free of umbilic points.  This is possible because $\Sigma$ is not a sphere. 
Let $H=c$. Then Eq. \eqref{eq1} writes as 
$$c|p|^2-\alpha\langle N(p),p\rangle=0,\quad p\in \Omega.$$
Let $\{e_1,e_2\}$ be an orthonormal frame in $\Omega$ formed by principal directions, where $Se_i=\kappa_i e_i$. 
Differentiating the above identity with respect to $e_1$ and $e_2$, we obtain 
\begin{equation*}
\begin{split}
(2c+\alpha\kappa_1)\langle e_1,p\rangle&=0,\\
(2c+\alpha\kappa_2)\langle e_2,p\rangle&=0.
\end{split}
\end{equation*}
Since $N(p)\not=\pm \frac{p}{|p|}$, we deduce that $\langle e_1,p\rangle\not=0$ or $\langle e_2,p\rangle\not=0$ in a subdomain $\widetilde{\Omega}\subset\Omega$.   Thus   we have $2c+\alpha\kappa_1=0$ in $\widetilde{\Omega}$, hence $\kappa_1$ is constant in $\widetilde{\Omega}$. This is a contradiction and this finishes the proof of Thm. \ref{t3}.

\section*{Acknowledgements}
The author  has been partially supported by MINECO/MICINN/FEDER grant no. PID2023-150727NB-I00,  and by the ``Mar\'{\i}a de Maeztu'' Excellence Unit IMAG, reference CEX2020-001105- M, funded by MCINN/AEI/10.13039/ 501100011033/ CEX2020-001105-M.

\end{document}